\documentclass[11pt]{amsart}
\usepackage{amssymb,amsmath,color}
\usepackage{tikz-dependency}
\usepackage{tikz}
\usepackage{enumerate}
\usepackage{mathtools}

\usetikzlibrary{snakes}

\def\multichoose
#1#2{\ensuremath{\left(\kern-.3em\left(\genfrac{}{}{0pt}{}{#1}{#2}\right)\kern-.3em\right)}}

\newtheorem{theorem}{Theorem}[section]
\newtheorem{proposition}[theorem]{Proposition}
\newtheorem{lemma}[theorem]{Lemma}
\newtheorem{definition}[theorem]{Definition}

\newtheorem{corollary}[theorem]{Corollary}
\newtheorem{example}[theorem]{Example}
\newtheorem{remark}[theorem]{Remark}

\numberwithin{equation}{section}

\newcommand{\Nnn}{{\mathbb N}}
\newcommand{\Ppp}{{\mathbb P}}

\newcommand{\ii}{\overline{\imath}}
\newcommand{\lv}{\overline{\ell}}
\newcommand{\FS}{{\mathcal FS}}

\newcommand{\vanish}[1]{}

\begin{document}

\title{Catalan--Spitzer permutations}

\author{Richard EHRENBORG, G\'abor HETYEI and Margaret READDY}

\address{Department of Mathematics, University of Kentucky, Lexington,
  KY 40506-0027.\hfill\break \tt http://www.math.uky.edu/\~{}jrge/,
  richard.ehrenborg@uky.edu.}

\address{Department of Mathematics and Statistics,
  UNC Charlotte, Charlotte NC 28223-0001.\hfill\break
\tt http://webpages.uncc.edu/ghetyei/,
ghetyei@uncc.edu.}

\address{Department of Mathematics, University of Kentucky, Lexington,
  KY 40506-0027.\hfill\break
\tt http://www.math.uky.edu/\~{}readdy/,
margaret.readdy@uky.edu.}



\begin{abstract}
We study two classes of permutations intimately related to the visual
proof of Spitzer's lemma and Huq's generalization of the Chung--Feller
theorem.  Both classes of permutations are counted by the Fuss--Catalan
numbers. The study of one class leads to a generalization of 
results of Flajolet from continued fractions to continuants. The study
of the other class leads to the discovery of a restricted variant of the
Foata--Strehl group action. 
\end{abstract}

\maketitle

\section*{Introduction}

A classical result of lattice path enumeration arising
from tossing $n$ fair coins
is the Chung-Feller
theorem~\cite{Chung_Feller}. It states that the Catalan number $C_n$ counts not
only the lattice paths consisting of unit northeast and southeast steps
from $(0,0)$ to $(2n,0)$ that stay above the horizontal axis, but we can
also prescribe the number $r$ of northeast 
steps above the horizontal axis. For each $r\in \{0,1,\ldots,n\}$ we
have the 
same Catalan number of lattice paths. Generalizations of this result are
due to Spitzer~\cite[Theorem 2.1]{Spitzer} as
well as Huq in his
dissertation~\cite[Theorem
  2.1.1]{Huq}. All these results may be shown using the following simple
visual idea: if we slightly ``tilt'' the diagram of a lattice path (see
Figure~\ref{figure_infinite_lattice_path}), all 
steps occur at different heights, and the relative order of these heights
may be rotated cyclically by changing the designation of the first step
in the lattice path. This simple idea was perhaps first used by
Raney~\cite[Theorem 2.1]{Raney}, who observed that there is exactly one
rotational equivalent of a sequence of $n+1$ positive units and $n$
negative units in which the partial sums are all positive. A question
naturally arises: which permutations are relative orders of steps
in such tilted pictures of lattice paths? 

In this paper we partially answer this general question in two specific
settings.  Both are related to {\em $k$-Catalan paths}, defined as lattice
paths consisting of unit up steps $(1,1)$ and down steps $(1,-k+1)$
which start and end on the horizontal axis but never go below it. 
The study of the relative
order of all steps leads us to a generalization of some results of Flajolet from
continued fractions to continuants. The study of the relative orders
of the up steps leads us to the discovery of a restricted variant of the
Foata--Strehl group action~\cite{Foata-Strehl1,Foata-Strehl2}.

Our paper is structured as follows. In the Preliminaries we review the
Chung-Feller theorem~\cite{Chung_Feller}, its generalizations by
Spitzer~\cite[Theorem 2.1]{Spitzer} and Huq~\cite[Theorem 2.1.1]{Huq},
and we point out a few connections between the two generalizations. In
Section~\ref{section_Huq_visual} we outline a visual proof of Huq's
results which inspires the definition of the permutations we intend to
study. We introduce {\em Catalan--Spitzer permutations} (and their
$k$-generalizations) in Section~\ref{section_Catalan_Spitzer} as the
relative orders of all steps in a Catalan path. Equivalently, these are obtained
by labeling the steps in reverse lexicographic order 
and listing them in the order they occur along the path. Due to this labeling, a
refined count of Catalan--Spitzer permutations amounts to enumerating
all Catalan paths that have a given number of steps at a certain level. 
For the Catalan paths our formulas may be obtained using Flajolet's
result~\cite[Theorem 1]{Flajolet} which provides a generalized continued 
fraction formula. We generalize these formulas to $k$-Catalan paths by
using continuants instead of continued fractions. In
Section~\ref{section_short_catalan_spitzer} we observe that the relative
order of the up steps alone uniquely determines the Catalan paths. The
resulting {\em short Catalan--Spitzer permutations} may be characterized
in terms of the associated {\em Foata--Strehl trees}, first studied by
Foata and Strehl~\cite{Foata-Strehl1,Foata-Strehl2} who introduced a
group action on the set of all permutations using these ordered $0-1-2$
trees. Finally, in Section~\ref{section_restricted_Foata_Strehl} we
study a restricted variant of the Foata--Strehl group action which takes
each short Catalan--Spitzer permutation into another short
Catalan--Spitzer permutation. The number of orbits on the set of $C_{n}$
permutations is the Catalan number $C_{n-1}$. This is a consequence of a
generating function formula that is applicable to any class of
permutations that is closed under the restricted Foata--Strehl group action. In
particular, for the set of all permutations the number of orbits is the
same as the number of indecomposable permutations.

Our results inspire revisiting three classical topics: generalizations
of the Chung--Feller theorem, Flajolet's continued fraction approach to
lattice path enumeration and the Foata--Strehl group actions. They are
likely the first to connect these three areas.

\section{Preliminaries}

This paper focuses on permutations that are associated to the {\em Chung and
Feller theorem}~\cite{Chung_Feller} and some of its generalizations.  
\begin{theorem}[Chung--Feller]
\label{theorem_Chung_Feller}
Among the lattice paths from $(0,0)$ to $(2n,0)$ consisting of $n$ up
steps $(1,1)$ and $n$ down steps $(1,-1)$, the number of paths having
$2r$ steps above the $x$ axis is the Catalan number
$C_{n}=\frac{1}{n+1}\binom{2n}{n}$, independently of $r$, for each $r\in
\{0,1,\ldots,n\}$.  
\end{theorem}  
In the special case when $r=n$ the Chung--Feller theorem implies that
the number of {\em Dyck paths}, that is, lattice paths of the above type
that never go below the $x$ axis from $(0,0)$ to $(2n,0)$ is the Catalan number
$C_n$. This well-known special case has been generalized to {\em $k$-Dyck
  paths} (whose definition may be found in Lemma~\ref{lemma_Raney} below) by
Raney; see~\cite[p.\ 361]{Graham_Knuth_Patashnik}.    
\begin{lemma}[Raney]
\label{lemma_Raney}  
The number of lattice paths from~$(0,0)$ to $(kn,0)$
consisting of $(k-1)n$ up steps $(1,1)$ and $n$ down steps $(1,1-k)$
that never go below the $x$-axis is the {\em Fuss--Catalan number}
\begin{align}
C_{n,k}
& =
\frac{1}{kn+1}\binom{kn+1}{n}
=
\frac{1}{(k-1)n+1}\binom{kn}{n}.
\end{align}  
\end{lemma}  
Huq has generalized Theorem~\ref{theorem_Chung_Feller} to the lattice
paths appearing in Lemma~\ref{lemma_Raney} by proving the following 
result~\cite[Theorem 2.1.1]{Huq}.
\begin{theorem}[Huq]
\label{theorem_Huq}  
Let $(y_{1},\ldots,y_{m})$ be any sequence of integers whose sum is $1$,
Then for each $r\in\{0,1,\ldots, m-1\}$ exactly one of the cyclic shifts
\begin{align*}
(y_{\sigma(1)},y_{\sigma(2)},\ldots,y_{\sigma(m)})
& \in
\{(y_{1}, y_{2}, \ldots, y_{m}),
(y_{2}, \ldots, y_{m}, y_{1}), \ldots,
(y_{m}, y_{1}, \ldots, y_{m-1})\} 
\end{align*}
has the property that exactly $r$ of the partial sums
$y_{\sigma(1)}+y_{\sigma(2)}+\cdots+y_{\sigma(k)}$ for $1 \leq k \leq m$
are positive. 
\end{theorem}
Huq's proof of Theorem~\ref{theorem_Huq} is a
consequence of the following theorem of Spitzer~\cite[Theorem 2.1]{Spitzer}.  
\begin{theorem}[Spitzer]
Let $(x_{1},x_{2},\ldots,x_{m})\in {\mathbb R}^m$ be any vector with
real coordinates such that 
\begin{align*}
x_{1}+x_{2}+\cdots+x_{m} &= 0
\end{align*}
but no shorter cyclic partial sum
$x_{i+1}+x_{i+2}+\cdots+x_{j} $ of the coordinates vanishes. Then for
each $r\in\{0,1,\ldots,m-1\}$ exactly one 
of the cyclic shifts
\begin{align*}
(x_{\sigma(1)},x_{\sigma(2)},\ldots,x_{\sigma(m)})
& \in
\{(x_{1}, x_{2}, \ldots, x_{m}),
(x_{2}, \ldots, x_{m}, x_{1}), \ldots,
(x_{m}, x_{1}, \ldots, x_{m-1})\} 
\end{align*}
has the property that exactly $r$ of the partial sums
$x_{\sigma(1)}+x_{\sigma(2)}+\cdots+x_{\sigma(k)}$
for $1 \leq k \leq m$ are positive.
\label{theorem_Spitzer}  
\end{theorem}
Indeed, introducing $x_i=y_i-1/m$ for $i=1,2,\ldots,m$, the resulting
vector $(x_{1},x_{2},\ldots,x_{m})\in {\mathbb R}^m$ satisfies the
conditions of Theorem~\ref{theorem_Spitzer} as the sum
$x_1+x_2+\cdots+x_m$ is zero but no shorter partial sum
$x_{i+1}+x_{i+2}+\cdots+x_{j} $ of the coordinates, read
cyclically, is an integer. Furthermore, any
shorter sum $x_{i+1}+\cdots+x_{j}$ is positive if and only if
$y_{i+1}+\cdots+y_{j}-1\geq 0$ holds, since $y_{i+1}+\cdots+y_{j}$ is an
integer and we have  
$y_{i+1}+\cdots+y_{j}-1< x_{i+1}+\cdots+x_{j}< y_{i+1}+\cdots+y_{j}$. 
\begin{proof}[Proof of Theorem~\ref{theorem_Spitzer}]
Introducing
\begin{align}
  z_{i}&= x_{1}+x_{2}+\cdots+ x_{i},
\label{equation_partial_sums}  
\end{align}  
all cyclically consecutive sums may be expressed as 
$x_{i+1}+x_{i+2}+\cdots+x_{j}=z_{j}-z_{i}$.
This is clear when $i\leq j$, and it is easy to prove when $i>j$ using
$z_{m}=0$.  Putting the numbers $z_{1}, z_{2},\ldots,z_{m}$ in
increasing order, for each $0 \leq r \leq m-1$ there is exactly
one $z_{i}$, the $(r+1)$st largest number, for which exactly $r$ of the
differences $z_{j}-z_{i}$ are positive.  
\end{proof}
As observed by Huq~\cite[Corollary 5.1.2]{Huq},
Theorem~\ref{theorem_Huq} has the following consequence.
\begin{corollary}[Huq]
\label{corollary_Huq}  
The number of lattice paths from~$(0,0)$ to $(kn,0)$
consisting of $(k-1)n$ up steps $(1,1)$ and $n$ down steps $(1,1-k)$
with exactly $r$ up steps below the $x$-axis is independent of $r$ for
$r\in \{0,1,\ldots,(k-1)n\}$ and is given by the Fuss--Catalan number
$C_{n,k}$. 
\end{corollary}
\vanish{
Indeed, let us associate to each lattice path~$p$ from~$(0,0)$
to $(kn,0)$ the augmented path~$p'$ from~$(0,0)$
to $(kn+1,1)$ obtained by prepending an initial up step $(1,1)$ to
$p$. The operation $p \longmapsto p'$ is a bijection between the set of all
lattice paths ${\mathcal P}(n,k,0)$ from~$(0,0)$ to $(kn,0)$ consisting
of $(k-1)n$ up steps $(1,1)$ and $n$ down steps $(1,1-k)$, and the set
${\mathcal P}(n,k,1)$ of all those lattice paths  from~$(0,0)$ to
$(kn+1,1)$, consisting of $(k-1)n+1$ up steps $(1,1)$ and $n$ down steps
$(1,1-k)$ that begin with an up step. Note that ${\mathcal P}(n,k,1)$ is
a subset of~${\mathcal Q}(n,k,1)$ the set of all lattice paths from~$(0,0)$
to $(kn+1,1)$ consisting of $(k-1)n+1$ up steps $(1,1)$ and $n$ down steps
$(1,1-k)$. If, however, we partition~${\mathcal Q}(n,k,1)$ into cyclic
shift equivalence classes, each such equivalence class contains exactly
$(k-1)n+1$ elements of ${\mathcal P}(n,k,1)$. Theorem~\ref{theorem_Huq} 
says that the number of lattice paths $p'$ with $r$ up steps below the
$x$ axis is the same in each equivalence class. }

\begin{remark}
{\em The special instance of Spitzer's theorem when $j=m-1$ is
  often called Spitzer's lemma; see~\cite[Lemma 10.4.3]{Krattenthaler}.
  The special instance of Corollary~\ref{corollary_Huq} when $j=m-1$ is also
a special case of Raney's theorem~\cite[Theorem 2.1]{Raney};
see~\cite[p.\ 359]{Graham_Knuth_Patashnik}.
}
\end{remark}  

As noted, Theorem~\ref{theorem_Huq} above
is a consequence of Spitzer's theorem~\ref{theorem_Spitzer}, but the
converse is also true.
\begin{proposition} Spitzer's theorem~\ref{theorem_Spitzer} is a
  consequence of Huq's theorem~\ref{theorem_Huq}. 
\end{proposition}
\begin{proof}
Assume that the sum of all coordinates
$(x_1,\ldots,x_m)\in {\mathbb R}^m$ is zero, and each of the shorter
cyclically consecutive sum of terms $x_{i+1}+\cdots+x_{j}$ is nonzero
and has sign $\varepsilon_{i,j}$. We may assume that there is a 
fixed positive integer $k>0$ such that all numbers
$x_i$ are rational of the special form
\begin{align}
x_i=\frac{m\cdot y_i-1}{m\cdot k} \text{ for some } y_i\in {\mathbb Z}.
\label{equation_special_form}
\end{align}
Indeed, we may perturb the coordinates of $(x_1,\ldots,x_m)$ as long as
they satisfy all $m(m-1)-2$ inequalities of the form
\begin{equation}
  \varepsilon_{i,j}\cdot (x_{i+1}+\cdots+x_j)>0,
\label{equation_polytope}
\end{equation}
together with the equation
\begin{equation}
  x_1+\cdots+x_m=0
  \label{equation_hyperplane}
\end{equation}  
in ${\mathbb R}^m$. The inequalities~\eqref{equation_special_form}
define an open subset of the hyperplane defined
by~\eqref{equation_hyperplane}. This subset is not empty as it contains
the vector we began with. Points whose coordinates are of the 
form given in~\eqref{equation_special_form} form a dense subset in the
hyperplane defined by~\eqref{equation_hyperplane}, hence we 
may replace $(x_1,\ldots,x_m)$ with a vector whose coordinates
are of the form given in~\eqref{equation_special_form} and that satisfy
the same inequalities. Similarly to the other implication, 
Theorem~\ref{theorem_Spitzer} now follows from
Theorem~\ref{theorem_Huq} after observing that each
shorter sum $x_{i+1}+\cdots+x_{j}$ satisfies the inequality 
$y_{i+1}+\cdots+y_{j}-1\leq k\cdot (x_{i+1}+\cdots+x_{j})<
y_{i+1}+\cdots+y_{j}$.
\end{proof}

\section{A lattice path visualization of Huq's result}
\label{section_Huq_visual}

\begin{figure}
\begin{center}
\newcommand{\xx}{1.0}
\newcommand{\cc}{0.16}
\tikzstyle{place}=[circle,draw=black,fill=black,thick,
                   inner sep=0pt,minimum size=1.5mm]
\tikzstyle{transition}=[rectangle,draw=black,fill=black,thick,
                        inner sep=0pt,minimum size=1.5mm]
\begin{tikzpicture}
\node (a) at ({0*\xx},{0*\xx}) [transition] {};
\node (b) at ({1*\xx},{3*\xx}) [place] {};
\node (c) at ({2*\xx},{1*\xx}) [place] {};
\node (d) at ({3*\xx},{2*\xx}) [place] {};
\node (e) at ({4*\xx},{-1*\xx}) [place] {};
\node (f) at ({5*\xx},{1*\xx}) [place] {};
\node (g) at ({6*\xx},{2*\xx}) [place] {};
\node (h) at ({7*\xx},{1*\xx}) [transition] {};

\draw[-,thick] (a) -- (b) -- (c) -- (d) -- (e) -- (f) -- (g) -- (h);

\node [below] at (a) {\small $(0,0)$};
\node [above] at (b) {\small $(1,3)$};
\node [below] at (c) {\small $(2,1)$};
\node [above] at (d) {\small $(3,2)$};
\node [below] at (e) {\small $(4,-1)$};
\node [left] at (f) {\small $(5,1)$};
\node [above] at (g) {\small $(6,2)$};
\node [below] at (h) {\small $(7,1)$};

\end{tikzpicture}
\end{center}
\caption{The lattice path associated with the vector
$\mathbf{y} =  (3,-2,1,-3,2,1,-1)$.}
\label{figure_lattice_path}
\end{figure}
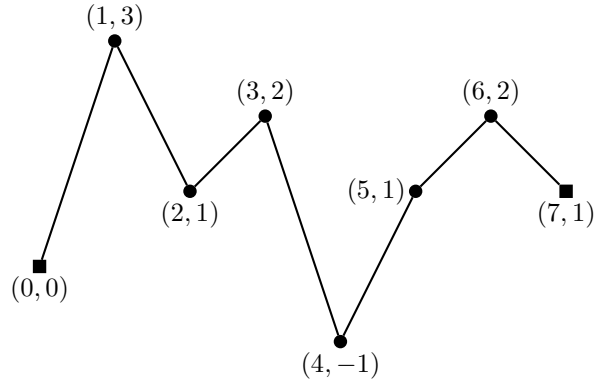

In the spirit of Krattenthaler~\cite[Remark 10.4.4]{Krattenthaler} and also
of Graham, Knuth and Patashnik~\cite[p.\ 360]{Graham_Knuth_Patashnik},
we may visualize a self contained proof of Theorem~\ref{theorem_Huq},
using lattice paths, as follows.
This visualization makes
the result and its proof a generalization of Raney's
lemma~\ref{lemma_Raney} and its geometric proof given
in~\cite[p.\ 359--360]{Graham_Knuth_Patashnik}. If we generalize
the notion of lattice paths to connect vertices with non-integer second
coordinates, our visualization also includes the proof of
Theorem~\ref{theorem_Spitzer}.

Let us extend  the vector $\mathbf{y}$ to an infinite vector $(\ldots,
y_{-1}, y_{0}, y_{1}, y_{2}, \ldots)$ 
by setting $y_{i} = y_{j}$ for $i \equiv j \bmod m$, and consider the
associated infinite lattice path with steps $\ldots, (1,y_{-1}),
(1,y_{0}), (1,y_{1}), (1,y_{2}),\ldots$, containing the lattice point 
$(0,v_{0}) = (0,0)$ and satisfying 
$(i+1,v_{i+1}) = (i,v_{i}) + (1,y_{i+1})$
for all integers $i$; see Figure~\ref{figure_infinite_lattice_path}.
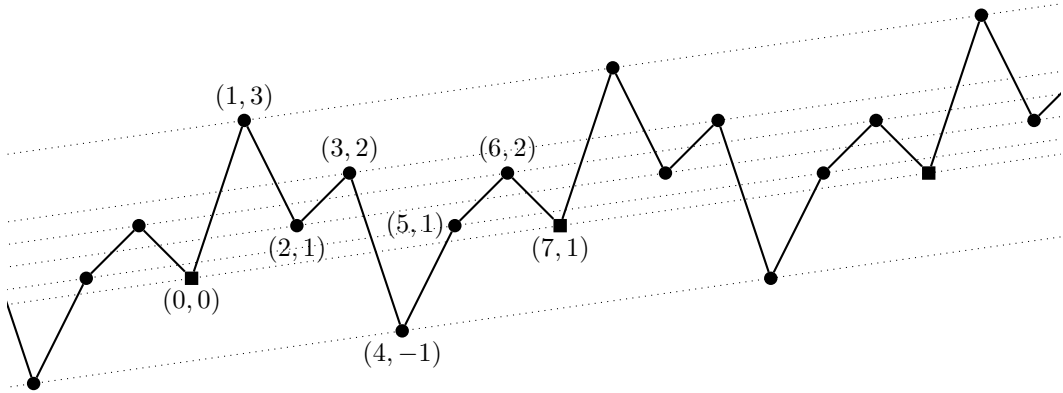
\begin{figure}[h]
\begin{center}
\newcommand{\xx}{0.7}
\newcommand{\cc}{0.16}
\tikzstyle{place}=[circle,draw=black,fill=black,thick,
                   inner sep=0pt,minimum size=1.5mm]
\tikzstyle{transition}=[rectangle,draw=black,fill=black,thick,
                        inner sep=0pt,minimum size=1.5mm]
\begin{tikzpicture}
\clip ({-3.5*\xx},{-2.5*\xx}) rectangle ({16.5*\xx},{5.5*\xx});
\node (ww) at ({-4*\xx},{1*\xx}) [place] {};
\node (xx) at ({-3*\xx},{-2*\xx}) [place] {};
\node (yy) at ({-2*\xx},{0*\xx}) [place] {};
\node (zz) at ({-1*\xx},{1*\xx}) [place] {};
\node (a) at ({0*\xx},{0*\xx}) [transition] {};
\node (b) at ({1*\xx},{3*\xx}) [place] {};
\node (c) at ({2*\xx},{1*\xx}) [place] {};
\node (d) at ({3*\xx},{2*\xx}) [place] {};
\node (e) at ({4*\xx},{-1*\xx}) [place] {};
\node (f) at ({5*\xx},{1*\xx}) [place] {};
\node (g) at ({6*\xx},{2*\xx}) [place] {};
\node (h) at ({7*\xx},{1*\xx}) [transition] {};
\node (i) at ({8*\xx},{4*\xx}) [place] {};
\node (j) at ({9*\xx},{2*\xx}) [place] {};
\node (k) at ({10*\xx},{3*\xx}) [place] {};
\node (l) at ({11*\xx},{0*\xx}) [place] {};
\node (m) at ({12*\xx},{2*\xx}) [place] {};
\node (n) at ({13*\xx},{3*\xx}) [place] {};
\node (o) at ({14*\xx},{2*\xx}) [transition] {};
\node (p) at ({15*\xx},{5*\xx}) [place] {};
\node (q) at ({16*\xx},{3*\xx}) [place] {};
\node (r) at ({17*\xx},{4*\xx}) [place] {};
\node (s) at ({18*\xx},{1*\xx}) [place] {};

\draw[-,thick] (ww) -- (xx) -- (yy) -- (zz) -- (a) -- (b) -- (c) -- (d) -- (e) -- (f) -- (g) -- (h)
-- (i) -- (j) -- (k) -- (l) -- (m) -- (n) -- (o) -- (p) -- (q) -- (r) -- (s);

\node [below] at (a) {\small $(0,0)$};
\node [above] at (b) {\small $(1,3)$};
\node [below] at (c) {\small $(2,1)$};
\node [above] at (d) {\small $(3,2)$};
\node [below] at (e) {\small $(4,-1)$};
\node [left] at (f) {\small $(5,1)$};
\node [above] at (g) {\small $(6,2)$};
\node [below] at (h) {\small $(7,1)$};

\draw[-, dotted] ({-7*\xx},{-1*\xx}) -- ({21*\xx},{3*\xx});
\draw[-, dotted] ({-6*\xx},{2*\xx}) -- ({22*\xx},{6*\xx});
\draw[-, dotted] ({-5*\xx},{0*\xx}) -- ({23*\xx},{4*\xx});
\draw[-, dotted] ({-11*\xx},{0*\xx}) -- ({24*\xx},{5*\xx});
\draw[-, dotted] ({-10*\xx},{-3*\xx}) -- ({25*\xx},{2*\xx});
\draw[-, dotted] ({-9*\xx},{-1*\xx}) -- ({26*\xx},{4*\xx});
\draw[-, dotted] ({-8*\xx},{0*\xx}) -- ({27*\xx},{5*\xx});

\end{tikzpicture}
\end{center}
\caption{The infinite lattice path associated to
the lattice path in Figure~\ref{figure_lattice_path}.
The dotted lines are the lines where the
functional $F(u,v) = v - 1/m \cdot u$ is constant.
The functional $F$ yields the linear order
$(4,-1)$, $(0,0)$, $(5,1)$, $(2,1)$, $(6,2)$, $(3,2)$ and $(1,3)$
on the points on the path.
}
\label{figure_infinite_lattice_path}
\end{figure}
The finite path $p_{i}$ occurs in the infinite path
as a subpath starting at $(i, v_{i})$ and ending at $(i+m, v_{i}+1)$.
Introducing $x_{i}=y_{i}-1/m$ and ordering the $z_{i}$s defined
in~\eqref{equation_partial_sums} amounts to the following. Consider the linear
functional $F(u,v) = v - 1/m\cdot u$ defined on the plane, and consider
its level curves which are lines with slope $1/m$. For any $i\in
\{1,2,\ldots, m\}$ we have $z_{i}=F(i,v_{i})$ and we may extend this
observation to all $i\in {\mathbb Z}$ keeping in mind that $z_{m}=0$. Thus
we may set $z_{i}=z_{j}$ if $i \equiv j \bmod m$. 
Ordering $z_{1}, z_{2},\ldots, z_{m}=z_{0}$ in increasing order amounts to
ordering  the $m$ lattice points $(0,v_{0})$ through $(m-1,v_{m-1})$
according to the linear functional $F$.  If $(i,v_{i})$
is the $(r+1)$st largest lattice point in this order, then there are
exactly $r$ lattice points $(i,v_{i})$ above this level.

\section{Catalan--Spitzer permutations}
\label{section_Catalan_Spitzer}

In this section we investigate the restriction of 
Theorem~\ref{theorem_Huq} and its proof to $k$-Catalan paths. In particular, we
describe the permutations of partial sums that appear in the proof of
Theorem~\ref{theorem_Huq}, when we prove it by reducing it to Spitzer's
theorem~\ref{theorem_Spitzer}. 

We define an {\em augmented} $k$-Catalan path of order $n$ as a lattice
path consisting of $(k-1)n+1$ up steps $(1,1)$ and $n$ down steps
$(1,-k+1)$ that begins with an up step and never goes below the line
$y=1$ after the initial up step. The sum of the second coordinates of
these steps is $1$, hence Theorem~\ref{theorem_Huq} is applicable.
In this special case the proof of this theorem calls for
replacing each step $(1,y)$ by $(1,y-1/(kn+1))$. Hence up 
steps become $(1,kn/(kn+1))$ and 
down steps become $(1,-k((k-1)n + 1)/(kn + 1))$ and the transformed
path goes from $(0,0)$ to $(kn+1,0)$. Between its endpoints it
remains strictly above the line $y=0$. The transformed path is not a
lattice path, but we can easily transform it to one by multiplying all
$y$-coordinates by the factor $(kn+1)/k$. This vertical stretch does not
change the relative vertical order of the $y$-coordinates of the
endpoints of the steps.

\begin{definition}
A {\em $k$-Catalan--Spitzer path of order $n$} is a lattice path consisting of
$(k-1)n+1$ up steps $(1,n)$ and $n$ down steps $(1,-((k-1)n + 1)$ from $(0,0)$
to $(kn+1,0)$ that remains strictly above the line $y=0$. 
\end{definition}

There is a natural bijection between augmented $k$-Catalan paths
and $k$-Catalan--Spitzer paths of order $n$: we associate 
to each augmented $k$-Catalan path $(0,0),(1,z_{1}),\ldots, (kn,
z_{kn}), (kn+1,1)$ the $k$-Catalan--Spitzer path
$(0,0),(1,z_{1}'),\ldots, (kn, z_{kn}'), (kn+1,0)$ in which up steps and
down steps follow in the same order. Hence the number of
$k$-Catalan--Spitzer paths of order $n$ is also the
Fuss-Catalan number $C_{n,k}$. The second coordinates $z_i'$ of the
lattice points in a  $k$-Catalan--Spitzer path pairwise differ and
may be easily computed from the second coordinates $z_i$ of the
corresponding augmented $k$-Catalan path as follows. Since a
$k$-Catalan--Spitzer path is obtained from the corresponding
augmented $k$-Catalan path by first decreasing the second coordinate of each
step by $1/(kn+1)$ and then performing a vertical stretch by a factor of
$(kn+1)/k$, we obtain that
\begin{align}
z_{i}'=\frac{(kn+1)\cdot z_{i}- i}{k} \text{ holds for } i=1,2,\ldots,
kn.
\label{equation_ztransform}
\end{align}
The next proposition
describes the relative position of these lattice points in a
$k$-Catalan--Spitzer path in  terms of the positions of lattice points in
the corresponding augmented $k$-Catalan path. 
\begin{proposition}
Consider an augmented $k$-Catalan path
$(0,0),(1,z_{1}),\ldots, (kn, z_{kn}), (kn+1,1)$ of order $n$, and
let $(0,0),(1,z_{1}'),\ldots, (kn, z_{kn}'), (kn+1,0)$ be the
corresponding $k$-Catalan--Spitzer path. Then for some $i\neq j$ the
inequality $z_{i}'<z_{j}'$ holds if and only if $(-i,z_{i}) < (-j, z_{j})$
in the reverse lexicographic order, where coordinates are compared right
to left. 
\label{proposition_revlexorder}
\end{proposition}  
\begin{proof}
By~\eqref{equation_ztransform} we have 
\begin{align}
z_{j}'-z_{i}'
&=\frac{(kn+1)\cdot (z_{j}-z_{i})+(i-j)}{k}.
\label{equation_zij}
\end{align}
Observe first that in the case when $z_{i}\neq z_{j}$, the sign of
$z_{j}'-z_{i}'$ is the same as the sign of $z_{j}-z_{i}$. Indeed,
$(kn+1)\cdot (z_{j}-z_{i})$ in~\eqref{equation_zij} above is a nonzero
integer multiple of $(kn+1)$, whereas $|i-j|<kn$.  
On the other hand, in the case when $z_{i}= z_{j}$,
by~\eqref{equation_zij} we have $z_{i}'<z_{j}'$ if 
and only if $-i< -j$ holds. 
\end{proof}
Proposition~\ref{proposition_revlexorder} inspires the following
definition.

\begin{definition}  
\label{definition_Catalan_Spitzer_permutation}
A {\em $k$-Catalan--Spitzer permutation of order $n$} is the relative
order of the numbers $z_{1}',\ldots,z_{kn}'$ in a $k$-Catalan--Spitzer path
$(0,0),(1,z_{1}'),\ldots, (kn, z_{kn}'), (kn+1,0)$ of order
$n$. Equivalently it is the relative order of the lattice points
$(1,z_{1}),\ldots, (kn, z_{kn})$ in the corresponding augmented
$k$-Catalan path of order $n$, where we order the lattice points first
by the second coordinate in increasing order and then by the first
coordinate in decreasing order. In the case when $k=2$ we will use the
term {\em Catalan--Spitzer permutation}.
\end{definition}

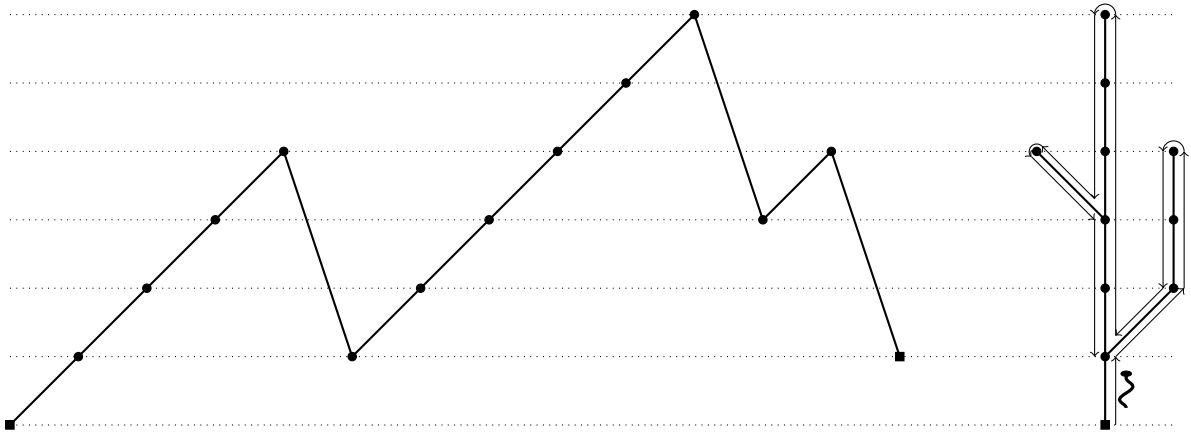
\begin{figure}
\begin{center}
\newcommand{\circlesize}{0.08}
\newcommand{\xx}{1.3}
\newcommand{\yy}{1.3}
\newcommand{\sss}{16}
\begin{tikzpicture}[scale = 0.7]
\foreach \j in {0,1, ..., 6} {\draw[dotted] ({0*\xx},{\j*\yy}) -- ({(\sss+1)*\xx},{\j*\yy});};

\draw[-,thick] ({0*\xx},{0*\yy})
-- ({1*\xx},{1*\yy})
-- ({2*\xx},{2*\yy})
-- ({3*\xx},{3*\yy})
-- ({4*\xx},{4*\yy})
-- ({5*\xx},{1*\yy})
-- ({6*\xx},{2*\yy})
-- ({7*\xx},{3*\yy})
-- ({8*\xx},{4*\yy})
-- ({9*\xx},{5*\yy})
-- ({10*\xx},{6*\yy})
-- ({11*\xx},{3*\yy})
-- ({12*\xx},{4*\yy})
-- ({13*\xx},{1*\yy})
;
\filldraw ({-\circlesize},{-\circlesize}) rectangle (\circlesize,\circlesize);
\filldraw ({1*\xx},{1*\yy}) circle (\circlesize);
\filldraw ({2*\xx},{2*\yy}) circle (\circlesize);
\filldraw ({3*\xx},{3*\yy}) circle (\circlesize);
\filldraw ({4*\xx},{4*\yy}) circle (\circlesize);
\filldraw ({5*\xx},{1*\yy}) circle (\circlesize);
\filldraw ({6*\xx},{2*\yy}) circle (\circlesize);
\filldraw ({7*\xx},{3*\yy}) circle (\circlesize);
\filldraw ({8*\xx},{4*\yy}) circle (\circlesize);
\filldraw ({9*\xx},{5*\yy}) circle (\circlesize);
\filldraw ({10*\xx},{6*\yy}) circle (\circlesize);
\filldraw ({11*\xx},{3*\yy}) circle (\circlesize);
\filldraw ({12*\xx},{4*\yy}) circle (\circlesize);
\filldraw ({13*\xx-\circlesize},{1*\yy-\circlesize}) rectangle ({13*\xx+\circlesize},{1*\yy+\circlesize});

\filldraw ({\sss*\xx-\circlesize},{0*\yy-\circlesize}) rectangle ({\sss*\xx+\circlesize},{0*\yy+\circlesize});
\filldraw ({\sss*\xx},{1*\yy}) circle (\circlesize);
\filldraw ({\sss*\xx},{2*\yy}) circle (\circlesize);
\filldraw ({\sss*\xx},{3*\yy}) circle (\circlesize);
\filldraw ({\sss*\xx},{4*\yy}) circle (\circlesize);
\filldraw ({\sss*\xx},{5*\yy}) circle (\circlesize);
\filldraw ({\sss*\xx},{6*\yy}) circle (\circlesize);
\filldraw ({(\sss+1)*\xx},{2*\yy}) circle (\circlesize);
\filldraw ({(\sss+1)*\xx},{3*\yy}) circle (\circlesize);
\filldraw ({(\sss+1)*\xx},{4*\yy}) circle (\circlesize);
\filldraw ({(\sss-1)*\xx},{4*\yy}) circle (\circlesize);
\draw[-,thick] ({\sss*\xx},{0*\yy}) -- ({\sss*\xx},{6*\yy});
\draw[-,thick] ({\sss*\xx},{1*\yy}) -- ({(\sss+1)*\xx},{2*\yy})  -- ({(\sss+1)*\xx},{4*\yy});
\draw[-,thick] ({\sss*\xx},{3*\yy}) -- ({(\sss-1)*\xx},{4*\yy});

\draw[->] ({\sss*\xx+0.2},{0*\yy}) -- ({\sss*\xx+0.2},{1*\yy});
\draw[->] ({\sss*\xx+0.2},{1*\yy}) -- ({(\sss+1)*\xx+0.2},{2*\yy});
\draw[->] ({(\sss+1)*\xx+0.2},{2*\yy}) -- ({(\sss+1)*\xx+0.2},{4*\yy});
\draw[->] ({(\sss+1)*\xx+0.2},{4*\yy}) arc (0:180:0.2);
\draw[->] ({(\sss+1)*\xx-0.2},{4*\yy}) -- ({(\sss+1)*\xx-0.2},{2*\yy});
\draw[->] ({(\sss+1)*\xx-0.2},{2*\yy}) -- ({(\sss)*\xx+0.2},{1*\yy+0.4});
\draw[->] ({(\sss)*\xx+0.2},{1*\yy+0.4}) -- ({(\sss)*\xx+0.2},{6*\yy});
\draw[->] ({(\sss)*\xx+0.2},{6*\yy}) arc (0:180:0.2);
\draw[->] ({(\sss)*\xx-0.2},{6*\yy}) -- ({(\sss)*\xx-0.2},{3*\yy+0.4});
\draw[->] ({(\sss)*\xx-0.2},{3*\yy+0.4}) -- ({(\sss-1)*\xx+0.1},{4*\yy+0.1});
\draw[->] ({(\sss-1)*\xx+0.1},{4*\yy+0.1}) arc (45:225:{sqrt(2)/10});
\draw[->] ({(\sss-1)*\xx-0.1},{4*\yy-0.1}) -- ({\sss*\xx-0.2},{3*\yy});
\draw[->] ({\sss*\xx-0.2},{3*\yy}) -- ({\sss*\xx-0.2},{1*\yy});

\draw [-,snake=snake, very thick] ({\sss*\xx+0.4},{0.25*\yy}) -- ({\sss*\xx+0.4},{0.75*\yy});
\filldraw ({\sss*\xx+0.4},{0.75*\yy}) ellipse (0.1 and 0.05);
\end{tikzpicture}
\end{center}
\caption{The augmented $4$-Catalan path whose associated
  $4$-Catalan--Spitzer permutation is $3,5,8,11,2,4,7,10,12,13,6,9,1$.}
\label{figure_4_Catalan_without_labelings_and_tree}
\end{figure}

\begin{example}
{\em An example of an augmented $4$-Catalan path and its labeling giving
  rise to the 
associated $4$-Catalan--Spitzer permutation is shown in
Figure~\ref{figure_4_Catalan_without_labelings_and_tree}. There are $3$
lattice points at level one, numbered right to left, followed by $2$ lattice
points at the next level, and so on. Modifying an idea presented
in~\cite{Stanley_EC2}, we may visualize an augmented $k$-Catalan path as a
description of the movement of a worm crawling around a rooted plane tree in
counterclockwise order, shown on the right of
Figure~\ref{figure_4_Catalan_without_labelings_and_tree}. The plane tree
is rooted with a root edge at level~$0$. Each up step in the lattice
path corresponds to the worm moving up one level and each down step
corresponds to the worm moving down $(k-1)$ levels. We can think of the
worm moving down $(k-1)$ times faster than up. The set of all rooted plane
trees with $n+1$ vertices is in bijection with the set of all augmented Catalan 
paths with $(n+1)$ up steps and $n$ down steps.  Only a subset of the set
of rooted plane trees with $(k-1)\cdot n+2$ vertices corresponds bijectively
to the set of $(k-1)$-Catalan paths with $(k-1)\cdot n+1$ up steps and
$n$ down steps. The numbering of the lattice points corresponds to the
labeling of the points where the worm begins or ends a move.   
}
\label{example_k_Catalan_Spitzer}
\end{example} 

In order to describe the finer structure of $k$-Catalan--Spitzer
permutations, we make the following definition.

\begin{definition}
Let $(i_1,\ldots,i_r)$ be a vector with nonnegative integer coordinates. 
We say that a $k$-Catalan--Spitzer permutation and the corresponding
augmented $k$-Catalan path has type
$(i_1,i_2,\ldots,i_r)$ if the augmented $k$-Catalan path has
$i_j$ lattice points at level $j$. We denote the number of
$k$-Catalan--Spitzer permutations having type $\ii = (i_1,i_2,\ldots,i_r)$ by
$t_{k}(\ii)$. 
\end{definition}
Note that this definition of type is not unique in the sense that an
augmented $k$-Catalan path has type $(i_1,i_2,\ldots,i_r)$ if and only if it
has type $(i_1,i_2,\ldots,i_r,0)$. In other words, we may add as many
zero coordinates to the type of an augmented $k$-Catalan path as we wish. 
We exclude the empty lattice path from consideration as we consider it
non-augmented. Hence $i_1$ must be positive.

Let $e_{i}$ denote the $i$th unit vector. For $S$ a finite subset of
positive integers, 
let $e_{S}$ denote the sum $e_{S}=\sum_{i \in S} e_{i}$
and $x_{S}$ denote the product $x_{S}=\prod_{i \in S} x_{i}$.
Furthermore, we also use the notation
$x^{\ii} = x_1^{i_1}x_2^{i_2}\cdots x_r^{i_r}$.
We write $[n] = \{1,2, \ldots, n\}$
and $[i,j] = \{i, i+1, \ldots, j\}$.

\begin{lemma}
\label{lemma_counting_types}  
The numbers $t_{k}(\ii)$,
where $\ii = (i_{1}, \ldots, i_{r-1},i_{r})$, are determined by 
the initial condition
$t_{k}(i_1) = \delta_{i_1,1}$ where $\delta_{k_1,1}$ is the Kronecker delta
and
if there is an index $j \in [r-k+1,r-1]$ such that $i_{j} < i_{r}$
then $t_{k}(\ii) = 0$.
When $r \geq k$ the following recurrence holds:
$$
t_{k}(\ii)
=
\binom{i_{r-k+1}-1}{i_r} \cdot t_{k}\left(\ii - i_{r} \cdot e_{[r-k+1,r]}\right) .
$$
\end{lemma}
\begin{proof}
In an augmented $k$-Catalan path exactly one lattice point must be at level $1$
if the lattice path never hits level $2$. This yields the initial condition.  
Notice that any lattice point at level $r$ in an augmented 
$k$-Catalan path of type $\ii$ is a peak immediately
preceded by a run of $k-1$ up steps $(1,1)$ and immediately followed by a
down step $(1,1-k)$. By removing these steps, each $k$-Catalan path of
type $\ii$ may be uniquely reduced
to an augmented $k$-Catalan path path of type
$$
(i_1,\ldots,i_{r-k}, i_{r-k+1}-i_{r},\ldots, i_{r-1}-i_{r}, i_{r}-i_r)
=  
\ii - i_{r} \cdot e_{[r-k+1,r]} .
$$  
Conversely, given an augmented $k$-Catalan path of type
$\ii - i_{r} \cdot e_{[r-k+1,r]}$,
there are
$$
\multichoose{i_{r-k+1}-i_r}{i_r}
=
\binom{i_{r-k+1}-1}{i_r}
$$
ways to select the place to reinsert $i_r$ runs of $k-1$ up steps
$(1,1)$ immediately followed by a down step $(1-k,1)$, after one of the
$i_{r-k+1}-i_r$ lattice points at level $r-k+1$ of the reduced lattice path,
where
$\multichoose{n}{j} = \binom{n+j-1}{j}$
denotes the number of $j$ element multisubsets of an $n$-set.
 \end{proof}  
Using Lemma~\ref{lemma_counting_types} we obtain the following
recurrence for the associated generating functions.

\begin{lemma}
\label{lemma_types_genf}  
The generating functions for the $k$-Catalan--Spitzer permutations of
type $\ii$, that is, 
$$
T_k(x_1,\ldots,x_r)
=
\sum_{\ii \in \Ppp \times \Nnn^{r-1}}
t_{k}(\ii) \cdot x^{\ii}
$$
are given by the initial conditions
$T_k(x_1)=T_k(x_1,x_2)=\cdots=T_k(x_1,\ldots,x_{k-1})=x_1$ and
for $r\geq k$ by the recurrence
\begin{equation}
\label{equation_trecurrence}
T_k(x_1,x_2,\ldots,x_r)
=T_k\left(x_1,x_2,\ldots,x_{r-k},\frac{x_{r-k+1}}{1-x_{[r-k+1,r]}},x_{r-k+2},\ldots,x_{r-1}
\right).
\end{equation}
\end{lemma}
Observe that the function on the left-hand side
of~(\ref{equation_trecurrence}) is $r$-ary, 
whereas the function on the right-hand side is $(r-1)$-ary.

\begin{proof}[Proof of Lemma~\ref{lemma_types_genf} .]
The initial conditions are straightforward to verify. Using the
recurrence stated in Lemma~\ref{lemma_counting_types},  we obtain
\begin{align}
T_k(x_1,\ldots,x_r)
& =
\nonumber
\sum_{\substack{\ii \in \Ppp \times \Nnn^{r-1} \\
i_{r}\leq i_{r-k+1},i_{r-k+2},\ldots,i_{r-1}}}
t_{k}\left(\ii - i_{r} \cdot e_{[r-k+1,r]}\right) \cdot  \binom{i_{r-k+1}-1}{i_r} \cdot x^{\ii} \\  
& =
\label{equation_t}
\sum_{\substack{\ii \in \Ppp \times \Nnn^{r-1} \\
i_{r}\leq i_{r-k+1},i_{r-k+2},\ldots,i_{r-1}}}
t_{k}\left(\ii - i_{r} \cdot e_{[r-k+1,r]}\right) \cdot x^{\ii - i_{r} \cdot e_{[r-k+1,r]}}
\cdot  \binom{i_{r-k+1}-1}{i_r} \cdot x_{[r-k+1,r]}^{i_r} .
\end{align}
Introduce $\lv$ to be the index vector $\ii - i_{r} \cdot e_{[r-k+1,r]}$
with the last zero removed,
that is, we set $\ell_{j} = i_{j}$ for $1 \leq j \leq r-k$
and
$\ell_{j} = i_{j}-i_r$ for $r-k+1 \leq j \leq r-1$.
The sum~(\ref{equation_t}) is now
\begin{align*}
T_k(x_1,\ldots,x_r)
& =
\sum_{\lv \in \Ppp \times \Nnn^{r-2}}
t_{k}(\lv) \cdot x^{\lv}
\cdot \sum_{0 \leq i_r} \binom{i_{r}+\ell_{r-k+1}-1}{i_r}
x_{[r-k+1,r]}^{i_{r}} \\
& =
\sum_{\lv \in \Ppp \times \Nnn^{r-2}}
t_{k}(\lv) \cdot x^{\lv
}
\cdot 
\frac{1}{(1- x_{[r-k+1,r]})^{\ell_{r-k+1}}} .
\qedhere
\end{align*}
\end{proof}

In order to give an explicit rational expression for these generating
functions, we define the denominator polynomial as follows. 
\begin{definition}
\label{definition_kcdenominator}
Given any positive integer $k\geq 2$ and any interval $[r,s]$ of
consecutive positive integers, we define the
{\em $k$-Catalan denominator polynomial $Q_k(x_r,x_{r+1},\ldots, x_s)$}
as the signed sum
$$
Q_k(x_r,x_{r+1},\ldots, x_s)
=
\sum_S (-1)^{{|S|}/{k}} \cdot x_{S},
$$
where $S$ ranges over all subsets of $[r,s]$ that arise as
a disjoint union of sets consisting of $k$ consecutive integers. The
empty set is included in the sum and contributes the term $1$. 
\end{definition}  

\begin{theorem}
\label{theorem_types_genf}  
  For $k\geq 2$ we have
  $$
T_k(x_1,x_2,\ldots,x_r)=\frac{x_1\cdot
  Q_k(x_2,\ldots,x_r)}{Q_k(x_1,\ldots,x_r)}. 
  $$
\end{theorem}  
\begin{proof}
For $r<k$ the sets $[r]$ and $[2,r]$ do not
contain any subset of $k$ consecutive integers, hence   
we have $Q_k(x_1,\ldots,x_r)=Q_k(x_2,\ldots,x_r)=1$ and
the identity holds. We proceed by induction for
$r\geq k$. By Lemma~\ref{lemma_types_genf} the generating function
$T_k(x_1,x_2,\ldots,x_r)$ is obtained from
$T_k(x_1,x_2,\ldots,x_{r-1})$ by substituting
$x_{r-k+1}/(1-x_{[r-k+1,r]})$ into $x_{r-k+1}$. This
substitution turns the stated formula for
$T_k(x_1,x_2,\ldots,x_{r-1})$ into a four-level fraction which may be
transformed into a quotient of two polynomials by multiplying the
numerator and the denominator by $1-x_{[r-k+1,r]}$.
This operation leaves all monomials
$x_{S}$ where $S\subseteq [r-1]$
containing a factor of $x_{r-k+1}$ unchanged, as replacing $x_{r-k+1}$ with
$x_{r-k+1}/(1-x_{[r-k+1,r]})$ and then multiplying by
$1-x_{[r-k+1,r]}$ amounts to no change at all.   
On the other hand, each monomial
$x_S$ satisfying $r-k+1\not\in S$ is replaced with
$x_S-x_S\cdot x_{[r-k+1,r]}$.

Assuming the induction hypothesis, the terms $x_S$ appearing
in the denominator of $T_k(x_1,\ldots,x_{r-1})$ are indexed exactly
by those subsets $S\subseteq [r-1]$ which arise as a disjoint
union of sets consisting of $k$ consecutive integers. Each such set 
is also a subset of $[r]$, and by our recurrence
the corresponding term $x_S$ remains in the denominator with the same
coefficient. The additional new terms in the denominator of
$T_k(x_1,x_2,\ldots,x_{r-1})$ are exactly the terms $x_S \cdot
x_{[r-k+1,r]}$ where $S$ ranges through all terms of
$Q_k(x_1,\ldots,x_{r-1})$ that do not contain $x_{r-k+1}$ as a factor. Note
that $r-k+1\not \in S\subseteq [r-1]$ implies $S\subseteq
[1,r-k]$ as the interval $[r-k+2,r-1]$ contains less
than $k$ consecutive integers. The converse is also true.
Hence the monomial $x_S \cdot x_{[r-k+1,r]}$ is
square-free, the underlying set is obtained by adding the disjoint set
$[r-k+1,r]$ consisting of $k$ consecutive integers to~$S$.
Hence we are adding exactly those terms of $Q_k(x_1,\ldots,x_{r})$
to the denominator which do not appear in $Q_k(x_1,\ldots,x_{r-1})$, and
the sign of $x_S \cdot x_{[r-k+1,r]}$ is the opposite of
$x_S$, consistent with the definition of $Q_k(x_1,\ldots,x_{r})$. 
Similar reasoning may be used to show that the numerator of
$T_k(x_1,x_2,\ldots,x_{r})$ is $x_1 \cdot Q_k(x_2,\ldots,x_{r})$.
\end{proof}  

\begin{example}
{\em 
As an example of Theorem~\ref{theorem_types_genf},  
we obtain for $k=3$ and $r=6$
\begin{align*}
T_3(x_1,x_2, \ldots, x_6)
& =
\frac{x_1\cdot (1 - x_{[2,4]} - x_{[3,5]} - x_{[4,6]})}
{1 - x_{[1,3]} - x_{[2,4]} - x_{[3,5]} - x_{[4,6]} + x_{[1,6]}} .
\end{align*}}
\end{example}

Theorem~\ref{theorem_types_genf} may be restated in a more compact form
in terms of the following generalization of continuants.

\begin{definition}
\label{definition_kcontinuant}  
The {\em $n$th $k$-continuant $K_{k,n}(x_1,x_2,\ldots, x_n)$} is defined
recursively by the initial condition
$K_{k,n} = x_{1} x_{2} \cdots x_{n}$ for $n < k$
and for $n \geq k$ by the recurrence
$$
K_{k,n}(x_1,x_2,\ldots,x_n)
=
K_{k,n-k}(x_1,x_2,\ldots,x_{n-k})+K_{k,n-1}(x_1,x_2,\ldots,x_{n-1})\cdot x_n .
$$
\end{definition}  
Note that for $k=2$ Definition~\ref{definition_kcontinuant} is the classical
definition of the continuants. The verification of the following facts
for $k$-continuants are completely analogous to the proof in the $k=2$
case, and are left to the reader. 
\begin{proposition}
\label{proposition_Kconsecutive}
The $k$-continuant $K_{k,n}(x_1,x_2,\ldots,x_n)$ can be computed by
taking the sum of all possible products of $x_1,x_2,\ldots,x_n$ in which
any number of disjoint sets of $k$ consecutive variables are deleted from
the product $x_1x_2\cdots x_n$.
\end{proposition}  
\begin{proposition}
  Introducing the $k\times k$ matrix
  $$
  M_k(x)=\begin{bmatrix}
  x & 0 & \ldots & 0 & 1\\
  1 & 0 & \ldots & 0 & 0\\
  0 & 1 & \ldots & 0 & 0\\
  \vdots & & \ddots & & \\
  0 & 0 & \ldots & 1 & 0\\
  \end{bmatrix},
  $$
  we may write
  $$
  \begin{bmatrix}
    K_{k,n}(x_1,x_2,\ldots,x_n)\\
    K_{k,n-1}(x_1,x_2,\ldots,x_{n-1})\\
    \vdots\\
    K_{k,n-k+1}(x_1,x_2,\ldots,x_{n-k+1})\\
  \end{bmatrix}
  =
  M_k(x_n)M_k(x_{n-1})\cdots M_k(x_1)
  \begin{bmatrix}
    1\\
    0\\
    \vdots\\
    0\\
  \end{bmatrix}.  
  $$
\end{proposition}
\begin{proposition}
The number of terms in $K_{k,n}(x_1,\ldots,x_n)$, that is, $K_{k,n}(1,\ldots,1)$,
is recursively obtained by
\begin{equation}
\label{equation_kcrecursion}
K_{k,n}(1,\ldots,1)
=
\begin{cases}
  1 & \text{ for $0\leq n\leq k-1$,} \\
  K_{k,n-k}(1,\ldots,1)+K_{k,n-k}(1,\ldots,1) & \text{ for } n\geq k.
\end{cases}
\end{equation}
\end{proposition}
For $k=2$ the sequence $K_{2,n}(1,\ldots,1)$ is the Fibonacci number $F_{n}$.
For $k=3$ the sequence $\{K_{3,n}(1,\ldots,1)\}_{n\geq 0}$ is sequence A000930
in~\cite{OEIS}, also known as Narayana's cow sequence. The same page
also contains information regarding the general sequence
$\{K_{k,n}(1,\ldots,1)\}_{n\geq 0}$.

As a direct consequence of Definition~\ref{definition_kcdenominator} and
Proposition~\ref{proposition_Kconsecutive}, we have the following corollary.
\begin{corollary}
\label{corollary_q_kc}
Given any positive integer $k\geq 2$, let $\zeta$ denote a
primitive $(2k)$th root of unity. Then for any interval $[r,s]$ of
consecutive positive integers, the $k$-Catalan
  denominator polynomial $Q_k(x_r,x_{r+1},\ldots, x_s)$ is given by
$$
Q_k(x_r,x_{r+1},\ldots, x_s)
=
\zeta^{s-r+1} \cdot x_r x _{r+1} \cdots x_s
\cdot K_{k,s-r+1}\left(\frac{1}{\zeta x_r},\frac{1}{\zeta x_{r+1}},\ldots,\frac{1}{\zeta x_s}\right). 
$$
\end{corollary}  

Substituting Corollary~\ref{corollary_q_kc} into
Theorem~\ref{theorem_types_genf}, after simplifying by $x_1x_2\cdots
x_r$, we obtain the formula
\begin{align}
\label{equation_types_continuants}  
T_k(x_1,x_2,\ldots,x_r)
& =
\frac{1}{\zeta}
\cdot
\frac{K_{k,r-1}\left(\frac{1}{\zeta x_2},\ldots,\frac{1}{\zeta x_r}\right)}
{K_{k,r}\left(\frac{1}{\zeta x_1},\ldots,\frac{1}{\zeta x_r}\right)}. 
\end{align}  

We conclude this section by having a closer look at the Catalan case.
Using the well-known relation between continuants and continued
fractions, equation~\eqref{equation_types_continuants} may be rewritten
as
\begin{align}
\label{equation_types}  
T_2(x_1,x_2,\ldots,x_r)
& =
\frac{1}{\mathbf{i}}
\cdot
\cfrac{1}{\frac{1}{\mathbf{i}\cdot
    x_1}+\cfrac{1}{\frac{1}{\mathbf{i}\cdot x_2}+\ddots
    \cfrac{1}{\frac{1}{\mathbf{i}\cdot x_r}}}}
=
\cfrac{1}{\frac{1}{x_1}-\cfrac{1}{\frac{1}{x_2}-\ddots \cfrac{1}{\frac{1}{x_r}}}} ,
\end{align}
where $\mathbf{i}$ is the square root of $-1$.

\begin{remark}
{\em Equation~\eqref{equation_types} is also a consequence of
  Flajolet's result~\cite[Theorem 1]{Flajolet} which provides a
  generating function formula for {\em Motzkin paths}, starting at
$(0,0)$ and ending on the $x$ axis consisting of up steps $(1,1)$, down
steps $(1,-1)$ and horizontal steps $(1,0)$ that never go below the $x$-axis.
To obtain Equation~\eqref{equation_types}, we must set $c_i=0$
and $a_{i} = b_{i} = x_{i+1}$ for all $i\geq 0$ in Flajolet's formula
and multiply the resulting generating function by $x_1$. The additional
factor of $x_1$ is induced by the fact that we consider augmented
Catalan paths.  Thus we obtain the generating function
$$
\cfrac{x_1}{1-\cfrac{x_1x_2}{1-\cfrac{x_2x_3}{1-\cfrac{x_3x_4}{\ddots}}}}.
$$
It is straightforward to see that we obtain the generating function that is the
limit, as $r$ goes to infinity, of the function given in~\eqref{equation_types}.}
\end{remark}

\section{Short $k$-Catalan--Spitzer permutations}
\label{section_short_catalan_spitzer}

The $k$-Catalan--Spitzer permutations defined in the previous section
contain redundant information. In this section we show that we may
restrict our attention to the lattice points which are the lower ends of
the up steps. The resulting permutations have a particularly
nice representation when we consider the Foata--Strehl group
action~\cite{Foata-Strehl1,Foata-Strehl2}.  

\begin{definition}  
A {\em short $k$-Catalan--Spitzer permutation of order $n$} is the relative
order of the numbers $z_{i_1}',\ldots,z_{i_{(k-1)n}}'$ in a
$k$-Catalan--Spitzer path 
$(0,0),(1,z_{1}'),\ldots, (kn, z_{kn}'), (kn+1,0)$ of order
$n$, where $\{i_{1},i_{2},\ldots,i_{(k-1)n}\}$ is the set of indices $i_j\geq
1$ satisfying $z_{i_j}'<z_{i_j+1}'$.  
In the case when $k=2$ we use the
term {\em short Catalan--Spitzer permutation}. 
\end{definition}

In analogy to Definition~\ref{definition_Catalan_Spitzer_permutation},
a short Catalan--Spitzer permutation may be also defined in terms of
$k$-Catalan paths as follows.

\begin{proposition}
\label{proposition_k_Catalan-short}  
The set of short $k$-Catalan--Spitzer permutations of order $n$ is the
set of all permutations arising by the following procedure. Take a
$k$-Catalan path of order $n$ and number its up steps so that
the values increase right to left at the same level and upward
between different levels. Record the numbers along the lattice path. 
\end{proposition}  
We will say that a short $k$-Catalan--Spitzer permutation $\sigma$ is {\em
  induced} by a $k$-Catalan path $P$ if the procedure described in
Proposition~\ref{proposition_k_Catalan-short} applied to $P$ yields the
permutation $\sigma$.

A short $k$-Catalan--Spitzer permutation associated to a
$k$-Catalan path may be computed directly from the
corresponding (full)  
$k$-Catalan--Spitzer permutation using the notions of an {\em ascents}
and {\em patterns}. Recall that the index $i\in \{1,\ldots, n-1\}$ is an
{\em ascent} of a permutation $\pi(1)\pi(2)\cdots\pi(n)$ if
$\pi(i)<\pi(i+1)$ holds. Furthermore, given an ordered alphabet $X$ with
$n$ letters, a {\em permutation} of $X$ is a word $w_1w_2\cdots w_n$
containing each letter of $X$ exactly once. The {\em pattern}
of the permutation $w$ is the permutation $\pi(1)\pi(2)\cdots \pi(n)$ of the set
$\{1,2,\ldots,n\}$ satisfying 
$\pi(i)<\pi(j)$ if and only if $w_i<w_j$ for each $1\leq i<j\leq n$.

The following lemma follows directly from the definitions. 

\begin{lemma}
Let $\pi(1)\pi(2)\cdots\pi(kn)$ and $\sigma(1)\sigma(2)\cdots
\sigma((k-1)n)$ be a $k$-Catalan--Spitzer, respectively a short
$k$-Catalan--Spitzer, permutation of order $n$ associated to the same
$k$-Catalan--Spitzer path. Then $\sigma(1)\sigma(2)\cdots
\sigma((k-1)n)$ is the pattern of
$\pi(i_{1})\pi(i_{2})\cdots\pi(i_{(k-1)n})$ where
$\{i_{1},i_{2},\ldots,i_{(k-1)n}\}$ is the set of ascents of
$\pi(1)\pi(2)\cdots\pi(kn)$.    
\end{lemma}  

\begin{example}
  {\em Consider the $4$-Catalan--Spitzer permutation
    $3,5,8,11,2,4,7,10,12,13,6,9,1$ of order $3$ discussed in
Example~\ref{example_k_Catalan_Spitzer}. Its ascent set is
$\{1,2,3,5,6,7,8,9,11\}$. The pattern of the word $3,5,8,2,4,7,10,12,6$
is $2,4,7,1,3,6,8,9,5$.}  
\end{example}  

As shown in Proposition~\ref{proposition_csp_injection} below, 
the operation assigning to each $k$-Catalan--Spitzer permutation
(equivalently, each $k$-Catalan path) the
corresponding short $k$-Catalan--Spitzer permutation is injective. We
will prove this by considering the {\em Foata--Strehl trees} of short
$k$-Catalan--Spitzer permutations. Recall that a {\em plane $0-1-2$
  tree} is a rooted plane tree, in which each vertex has degree at most
$2$. (It is not unusual to call plane $0-1-2$ trees plane binary trees.
However, a plane binary tree in the strict sense cannot contain a
vertex with a single child.) 

\begin{definition}
Let $w_{1}w_{2}\cdots w_{n}$ be a word with letters from an
ordered alphabet containing no repeated letters. The {\em Foata--Strehl
  tree $\FS(w_{1}w_{2}\cdots w_{n})$} of this word is the plane 
$0-1-2$ tree defined recursively as follows. The root of the tree is
$w_i=\min(w_{1},w_{2}, \ldots, w_{n})$ whose left child is
$\min(w_{1},w_{2}, \ldots, w_{i-1})$ and whose right child is
$\min(w_{i+1},w_{2}, \ldots, w_{n})$. There is no left child if $i=1$ and
no right child if $i=n$. The subtree of the left child is
$\FS(w_{1}w_{2}\cdots w_{i-1})$, and the subtree 
of the right child is $\FS(w_{i+1}w_{i+2}\cdots
w_{n})$.     
\end{definition}  

Clearly, the correspondence between permutations of $\{1,2,\ldots,n\}$
and Foata--Strehl trees with $n$ vertices is a bijection. 

\begin{lemma}
  \label{lemma_levels}
  Let $\sigma=\sigma(1)\sigma(2)\cdots 
\sigma((k-1)n)$ be a  short $k$-Catalan--Spitzer permutation induced by
a $k$-Catalan path $P$ of order $n$. 
  If $\sigma(i)$ is the label of an up step that is immediately followed
  by a down step in $P$ then $\sigma(i)$ has no right child.
  If $\sigma(i)$ is the label of an up step that is immediately followed
  by an up step in $P$ then $\sigma(i)$ has a right child $\sigma(j)$
  and the level of the up step labeled by $\sigma(j)$ is one more than
  the level of the up step labeled by $\sigma(i)$.
\end{lemma}
\begin{proof}
If the up step labeled $\sigma(i)$ is immediately followed by a down
step, then the level of the next up step is not greater than that of the
up step labeled $\sigma(i)$. In this case $\sigma(i+1)<\sigma(i)$ holds
and the right subtree of $\sigma(i)$ in $\FS(\sigma)$ is empty. Assume
from now on that 
the up step labeled $\sigma(i)$ is immediately followed by an up step.   
As we follow the up steps along $P$ after the up step labeled
$\sigma(i)$, all have a larger label than $\sigma(i)$ until $P$ returns
to the level of $\sigma(i)$. The label of the next up step is less than
the label of $\sigma(i)$. Hence the labels belonging to the right
subtree of $\sigma(i)$ in $\FS(\sigma)$ are exactly the up steps
belonging to the part $P'$ of $P$ that begins with the up step labeled
$\sigma(i)$ and ends with the first return of $P$ to the same level. The
labels in this subtree belong to up steps whose level is greater than
the level of the up step labeled $\sigma(i)$. The level of the
last up step in $P'$ is one more than that of the up step labeled
$\sigma(i)$ and it is the rightmost among all up steps of $P'$ at this
level. Hence its label is the least element of the right subtree of
$\sigma(i)$.   
\end{proof}
Inspired by Lemma~\ref{lemma_levels}, we define the {\em level} of each
letter in a permutation as follows.
\begin{definition}
  \label{definition_levels}
Let $T$ be a plane $0-1-2$ tree. We define the {\em level} of each
vertex of $T$ as follows.
\begin{enumerate}
\item The level of the root is zero.
\item For any other vertex $v$, the level of $v$ is the number
  of right steps in the unique path in $T$ from the root to $v$.
\end{enumerate}  
Given any permutation $\sigma=\sigma(1)\sigma(2)\cdots\sigma(n)$ we
define the level of $\sigma(i)$ as the level of the vertex labeled
$\sigma(i)$ in the the Foata--Strehl tree $\FS(\sigma)$.
\end{definition}  
Using Definition~\ref{definition_levels} we may rephrase
Lemma~\ref{lemma_levels} as follows.
\begin{corollary}
Let $\sigma=\sigma(1)\sigma(2)\cdots \sigma((k-1)n)$ be a short
$k$-Catalan--Spitzer permutation induced by a $k$-Catalan
path $P$ of order $n$. Then for each $i$ the level of the up step
labeled $\sigma(i)$ is the same as the level of $\sigma(i)$. 
\label{corollary_levels}
\end{corollary}  
An important consequence of Corollary~\ref{corollary_levels} is the
following. 
\begin{proposition}
The operation associating to each $k$-Catalan path $P$ its induced
$k$-Catalan--Spitzer permutation is injective. 
\label{proposition_csp_injection}  
\end{proposition}  
\begin{proof}
Assume the $k$-Catalan--Spitzer permutation $\sigma$ is induced by the
$k$-Catalan path $P$. It suffices to show that $P$ may be uniquely
reconstructed from the Foata--Strehl tree $\FS(\sigma)$ of $\sigma$. By
Corollary~\ref{corollary_levels} the level of each up step may be read
from $\FS(\sigma)$. Note finally that the number of down steps between
the up step labeled $\sigma(i)$ and the next up step labeled
$\sigma(i+1)$ is the difference between the level of $\sigma(i)$ and the
level of $\sigma(i+1)$.
\end{proof}  
Definition~\ref{definition_levels} allows us to define the level of each
letter in any permutation. The next definition allows us to identify
the short $k$-Catalan--Spitzer permutations by looking at their
Foata--Strehl trees.  
\begin{definition}
\label{definition_levelwise}  
  Let $T$ be a plane $0-1-2$ tree with $n$ vertices numbered from  $1$
  to $n$. We say the that $T$ is {\em levelwise numbered}  if the
  labeling of its vertices satisfies the following criteria:
  \begin{enumerate}
\item If the level of the vertex labeled $i$ is less than the level of
  the vertex labeled $j$ then $i<j$.
\item If the vertex labeled $j$ is in the left subtree of the vertex
  labeled $i$ then $i<j$.
\item If the vertex labeled $i$ and the vertex labeled $j$ have the same
  level, but there is a $k<i,j$ such that vertex labeled
  $i$ (respectively $j$) is in the right (respectively left) subtree of
  the vertex labeled~$k$ then $i<j$. 
\end{enumerate}  
\end{definition}  
\begin{proposition}
Each plane $0-1-2$ tree has a unique levelwise numbering. 
\label{proposition_unique}
\end{proposition}  
\begin{proof}
By Definition~\ref{definition_levelwise} vertices at the same level must
be numbered consecutively. We only need to show that the second and
third rules of the definition uniquely determine the labeling of the vertices
at the same level. We will show this by
considering for each vertex $v$ the unique path from the root to
$v$. This path may be encoded by an {\em $r\ell$-word $RL(v)$} defined
as follows. As we move along the path from the root to a vertex $v$, we
record a letter $r$ each time we move to the right child and a letter
$\ell$ each time to a left child. For example, $r\ell\ell$ represents
the left child of the left child of the right child of the root. Clearly
the level of $v$ is the number of letters $r$ in its $r\ell$-word. It
suffices to show that for vertices at the same level, the second and
third rules amount to ordering their $r\ell$-words in the left-to-right
lexicographic order as follows: the letter $r$ precedes the letter
$\ell$ and each word is succeeded by all words obtained by appending any
number of letters $\ell$ at their right end.

Consider the $r\ell$-words
$RL(u)$ and $RL(v)$ encoding the vertices $u$ and $v$ at the same
level. Let us compare their letters left to right and stop where we see
the first difference. Two possibilities arise:
\begin{enumerate}
\item One of the two words, say $RL(u)$, is an initial segment of the
  other. Since $u$ and $v$ are at the same level, in this case we must
  have $RL(v)=RL(u)\ell^k$ for some $k$. In this case $v$ is in the left
  subtree of $u$ and by the second rule $v$ must have a higher number
  than $u$.
\item Neither of the two words is an initial segment of the other one,
  and (without loss of generality) the leftmost letter that is different
  is $r$ in $RL(u)$ and $\ell$  in $RL(v)$. In other words, we have
  $RL(u)=RL(w)r x$ and $RL(v)=RL(w)\ell y$ for some $r\ell$-words $x$
  and $y$ and a vertex $w$ whose $r\ell$-word is the longest common
  initial segment of $RL(u)$ and $RL(v)$. In this case $u$ is in the
  right subtree of $w$, $v$ is in the left subtree of $w$ and the third
  rule is applicable. 
\end{enumerate}
Note that the above two cases, as well as the premises of the second and
third rules in Definition~\ref{definition_levelwise}, are mutually
exclusive and represent a complete enumeration of all possibilities. We
have shown that the rules in Definition~\ref{definition_levelwise} are
logically equivalent to defining the above lexicographic order on the
$r\ell$-words of the vertices at the same level. 
\end{proof}

\begin{proposition}
The Foata-Strehl tree of a short $k$-Catalan--Spitzer permutation of
order $n$ is levelwise numbered.
\label{proposition_fs_levelwise}  
\end{proposition}  
\begin{proof}
By Lemma~\ref{lemma_levels} the Foata--Strehl tree of a short
$k$-Catalan--Spitzer permutation $\sigma$ satisfies the first condition of
Definition~\ref{definition_levelwise}. Consider two letters $\sigma(i)$
and $\sigma(j)$ at the same level. By Lemma~\ref{lemma_levels} there is
a $k$-Catalan path $P$ inducing $\sigma$ in which $\sigma(i)$
and $\sigma(j)$ label up steps at the same level. If the vertex labeled
$\sigma(j)$ is in the left subtree of the vertex labeled $\sigma(i)$ then
$\sigma(j)$ precedes $\sigma(i)$ in $\sigma$, that is, we have $j<i$ and
$\sigma(i)<\sigma(j)$ must hold as up steps at the same level are
numbered in the right to left order. Similarly, if there is a vertex
labeled $\sigma(k)$ such that the vertex labeled $\sigma(i)$,
respectively $\sigma(j)$, is in its right, respectively left subtree,
then we must have $j<k<i$ and $\sigma(i)<\sigma(j)$ must hold. 
\end{proof}  

We conclude this section with a theorem completely describing short
$k$-Catalan--Spitzer permutations in terms of their Foata--Strehl trees.

\begin{theorem}
A permutation $\sigma=\sigma(1)\sigma(2)\cdots \sigma(n)$ is a short
Catalan--Spitzer permutation if and only if its Foata--Strehl tree
$\FS(\sigma)$ is levelwise numbered. It is also a short
$k$-Catalan--Spitzer permutation if and only if its Foata--Strehl tree
$\FS(\sigma)$ also has the following additional property: in each
longest path containing only edges between parent and right child, the
number of vertices is a multiple of $k-1$.   
\label{theorem_FS_characterization}
\end{theorem}
\begin{proof}
By Proposition~\ref{proposition_csp_injection} the number of short
Catalan--Spitzer permutations of order $n$ is the Catalan number
$C_n$. By Proposition~\ref{proposition_fs_levelwise} the Foata--Strehl
tree $\FS(\sigma)$ of each short Catalan--Spitzer permutation must be
levelwise numbered. By Proposition~\ref{proposition_unique} each plane
$0-1-2$ tree has exactly one levelwise numbering, and the number of plane
$0-1-2$ trees on $n$ vertices is also the Catalan number~$C_n$.
Therefore the set of Foata--Strehl trees of all short
Catalan--Spitzer permutations of order $n$ must equal the set of all
levelwise numbered plane $0-1-2$ trees on $n$ vertices.    

To prove the second statement, observe that each $k$-Catalan path of
order $n$ may be transformed into a Catalan path of order $(k-1)n$ by
replacing each down step $(1,1-k)$ with a run of $(k-1)$ consecutive
down steps $(1,-1)$. Under this correspondence $k$-Catalan paths of
order $n$ bijectively correspond to those Catalan paths of order
$(k-1)n$ in which the length of each longest run of consecutive down
steps is a multiple of $(k-1)$. Using this correspondence it is easy to
see that a short Catalan--Spitzer permutation of order $n$ is also a short
$k$-Catalan--Spitzer permutation of order $n$ if and only if for each
$i\in \{1,2,\ldots,n-1\}$ the difference between the level of
$\sigma(i)$ and the level of $\sigma(i+1)$ is a multiple of $k-1$. This
condition is equivalent to the condition stated in the theorem. The
details are left to the reader. 
\end{proof}  

\section{A restricted Foata--Strehl group action and its enumerative
  consequences} 
\label{section_restricted_Foata_Strehl}

Foata--Strehl trees were first
defined~\cite{Foata-Strehl1,Foata-Strehl2} to introduce the {\em
  Foata--Strehl group action} on permutations of order $n$. This
${\mathbb Z}_2^{n-1}$ action is generated by $n-1$ commuting involutions
$\phi_1,\phi_2,\ldots,\phi_{n-1}$, where $\phi_i$ swaps the left and
right subtrees of the vertex labeled $i$ in the Foata--Strehl tree of
each permutation. (Both subtrees may be empty.)
Theorem~\ref{theorem_FS_characterization} provides a characterization of
short $k$-Catalan--Spitzer permutations in terms of their Foata--Strehl
trees. Unfortunately in most cases the Foata--Strehl action destroys the
levelwise numbering, except for some special situations. In this
section we focus on such a special situation, introduce a subgroup of
the Foata--Strehl group action, and observe how it may be used in
proving identities in enumerative combinatorics beyond the world
of Catalan objects.

\begin{definition}
Let $X$ be an ordered alphabet, $x\in X$ and $w$ a permutation of
$X$. We say that $w$ is $x$-decomposable if it may be written in the
form $w=w_1w_2w_3$ such that the following hold:
\begin{enumerate}
\item All letters of $w_1$ and $w_3$ are less than $x$;
\item all letters of $w_2$ are greater than or equal to $x$;
\item the letter $x$ is either the first or the last letter of $w_2$.
\end{enumerate}
Under the above circumstances we call the decomposition $w=w_1w_2w_3$
the {\em $x$-decomposition} of $w$. An {\em $x$-flip} consists of
moving the letter $x$ from one end of the subword $w_2$ to its other
end.   
\end{definition}
When $x=\max(X)$ is the largest element of the alphabet $X$, any
permutation is {\em 
  trivially $x$-decomposable} since $w_2$ must consist of the single
letter $x$. For this $x$, the $x$-flip is the identity map. For all
other $x\in X$ an $x$-decomposition must be {\em nontrivial} as $w_2$
must contain all letters larger than $x$. The application of an $x$-flip
results in a different word.  The verification of the following
equivalent description is left to the reader.

\begin{proposition}
\label{proposition_decomposition_FS}
For $x<\max(X)$ the permutation $w$ is $x$-decomposable if and only if
its Foata--Strehl tree $\FS(w)$ satisfies the following:
\begin{enumerate}
\item exactly one of the right and left subtrees of $x$ is empty; 
\item the set of descendants of $x$ contains all letters larger than $x$.  
\end{enumerate}
If $w$ has an $x$-decomposition, an $x$-flip is precisely the
application of the operation $\phi_x$ of the Foata--Strehl group
action. 
\end{proposition}  
An important consequence of Proposition~\ref{proposition_decomposition_FS}
is that $x$-flips and $y$-flips commute the same way the generators of
the Foata--Strehl group action commute.
\begin{corollary}
If a word $w$ is simultaneously $x$-decomposable and $y$-decomposable
then the same holds for $\phi_x(w)$, $\phi_y(w)$ and
$\phi_x\phi_y(w)$. Furthermore,  $\phi_x\phi_y(w)=\phi_y\phi_x(w)$
\end{corollary}  

\begin{definition}
  Given an ordered alphabet $X$, we extend the $x$-flip operations $\phi_x$ to
  all permutations $w$ of $X$ by setting $\phi_x(w)=w$ whenever $w$ is
  not $x$-decomposable. We call the group action generated by the
  operations $\{\phi_x\: x<\max(X)\}$ the {\em restricted Foata--Strehl
    group action} on the permutations of~$X$. 
\end{definition}  

Our interest in $x$-decompositions and $x$-flips is due to the following
result.

\begin{theorem}
If a short Catalan--Spitzer permutation $\sigma$ of order $n$ is
$i$-decomposable for some $i<n$ then its $i$-flip $\phi_i(\sigma)$ is
also a short Catalan--Spitzer permutation.
\label{theorem_csp_orbits}
\end{theorem}  
\begin{proof}
By Proposition~\ref{proposition_decomposition_FS} the set of descendants of
$i$ is the set $\{i+1,i+2,\ldots,n\}$ and all of them are in the subtree
of $i+1$, which is the only child of $i$. If $i+1$ is the right child of
$i$ then the level of $i+1$ is one more than the level of $i$, $i$ is
the largest letter at its level and $i+1$ is the smallest letter at its
level. Moving $i+1$ to the left of $i$ results in merging the levels of
$i$ and $i+1$ into a single level: the level of the labels larger than
$i+1$ uniformly decrease by one. For all other labels, the sets of
labels at the same level remains unchanged. Consider two vertices
$u$ and $v$ whose label belongs to the merged level. If the labels of
$u$ and $v$ are both greater than $i$, then the shortest path leading to
both contains the vertex $i+1$: moving $i+1$ to the left induces
changing a letter $r$ into a letter $\ell$ in the same position in both
$RL(u)$ and $RL(v)$. Such a change does not affect the relative order of
the two $r\ell$-words in the lexicographic order. If the labels of
$u$ and $v$ are both at most $i$, then $RL(u)$ and $RL(v)$ remain
unchanged after moving $i+1$ to the left of $i$. Consider finally the
case when the label of $u$ is at most $i$ and the label of $v$ is less
than $i$. When the vertex labeled $i+1$ is the right child of the vertex
labeled $i$ then the label of $v$ is more than the label of $u$ because
they are at different levels. When we move the vertex labeled $i+1$ to
the left, the vertex $v$ becomes a vertex in the left subtree of the
vertex labeled $i$, hence its label must be still greater than $i$ and
also greater than the label of the vertex $u$ (whose $r\ell$-word
remains unchanged).

We have shown that applying an $i$-flip to an $i$-decomposable
Catalan-Spitzer permutation that moves $i+1$ from the right to the left
results in a Catalan-Spitzer permutation. The proof of the converse is
analogous and left to the reader. 
\end{proof}  

As a consequence of Theorem~\ref{theorem_csp_orbits} the set of all
Catalan-Spitzer permutations of order $n$ may be partitioned into
orbits of the restricted Foata--Strehl group action. A question
naturally arises, namely, what is the number of such 
orbits. We answer this in the greatest generality. A {\em class
  of permutations} $\mathcal P$ is a rule assigning to each finite
ordered set $X$ a subset ${\mathcal P}_{|X|}(X)$ of its permutations in
such a way that membership of $w$ in ${\mathcal P}_{|X|}(X)$ depends only
on the pattern of $w$. For brevity, we will say that a permutation $w$
{\em belongs to the class ${\mathcal P}$} if $w$ is an element of
${\mathcal P}_{|X|}(X)$ for some finite set~$X$. 
\begin{definition}
A class of permutations is {\em compatible with the
  restricted Foata--Strehl group action} if it satisfies the following:
for each $x\in X$ and  for each $x$-decomposable $w\in {\mathcal
  P}_{|X|}(X)$ the following holds:
\begin{enumerate}
\item The permutation $\phi_x(w)$ belongs to to the class ${\mathcal
  P}$. 
\item If $w_1w_2w_3$ is the $x$-decomposition of $w$ then $w_2$ and
  $w_1w_3$ also belong to the class ${\mathcal P}$. 
\end{enumerate}
\end{definition}
For a class of permutations ${\mathcal P}$ that is compatible with the
restricted Foata--Strehl group action, let~$P_n$ respectively $O_n$, be
the number of permutations in ${\mathcal P}(\{1,2,\ldots,n\})$,  respectively
number of orbits of the restricted Foata--Strehl group action on the set
${\mathcal P}(\{1,2,\ldots,n\})$. We introduce the ordinary generating
functions
\begin{equation}
P(x)=\sum_{n\geq 1} P_n\cdot x^n \quad \mbox{and}\quad Q(x)=\sum_{n\geq
  1} Q_n\cdot x^n.
\label{equations_gfns}
\end{equation}  
We consider these generating functions as formal power series from
${\mathbb Q}[[x]]$. Our first general result is the following.
\begin{theorem}
The generating functions $P(x)$ and $O(x)$ satisfy 
\begin{equation}
\label{equation_po}  
P(x)=\frac{O(x)}{1-O(x)}
\end{equation}
equivalently,  
\begin{equation}
\label{equation_op}  
O(x)=\frac{P(x)}{1+P(x)}
\end{equation}
\label{theorem_op}
\end{theorem}  
\begin{proof}
We only need to show~(\ref{equation_po}) as equation~(\ref{equation_op})    
is algebraically equivalent. 
For each $\sigma\in {\mathcal P}(\{1,2,\ldots,n\})$ denote the
orbit of $\sigma$ under the restricted Foata--Strehl group action by
$[\sigma]$. To each orbit $[\sigma]$ we may associate a set $I([\sigma])\subseteq
\{1,2,\ldots,n-1\}$ such that each permutation in the orbit is
$i$-decomposable if and only if $i$ is an element of $I([\sigma])$. The
size of the orbit will be $2^{|I([\sigma])|}$. We say that 
$\sigma$ is a {\em distinguished orbit representative} if for each
$i\in I$, the letter $i+1$ is to the right of 
$i$ in $\sigma$. Equivalently, in the Foata--Strehl tree $\FS(\sigma)$
of $\sigma$,  each $i\in I([\sigma])$ has a right child and not a left
child. Clearly there is exactly one distinguished orbit representative
in each orbit. For all permutations $\sigma\in {\mathcal
  P}(\{1,2,\ldots,n\})$ that are not distinguished orbit
representatives, there is a unique smallest $k\in I([\sigma])$ such that $k+1$ is
the left child of $k$ and $i+1$ is a right child for all $i\in I([\sigma])$
satisfying $i<k$. The removal of the left subtree of $k+1$ results in
the Foata--Strehl tree of a distinguished orbit representative in
${\mathcal P}(\{1,2,\ldots,k\})$, whereas the left 
subtree of $k$ is the Foata--Strehl tree of a permutation in ${\mathcal
  P}(\{k+1,k+2,\ldots,n\})$. The two permutations may be selected
independently and determine $\sigma$ uniquely. This observation
justifies the formula
$$
P_n=O_n+\sum_{k=1}^{n-1} O_{k}\cdot P_{n-k}. 
$$
The stated formula for the generating functions follows immediately. 
\end{proof}  
\begin{example}
{\em If ${\mathcal P}$ is the class of short Catalan-Spitzer permutations
then $P(x)=C(x)-1$ where
$$C(x)=
\sum_{n \geq 0} C_{n} \cdot x^{n} = \frac{1-\sqrt{1-4x}}{2x}
$$
is the generating function of the Catalan numbers. Equation~(\ref{equation_op})
gives
$O(x)=\frac{C(x)-1}{C(x)}=x\cdot C(x)$.
Hence $O_n=C_{n-1}$.
 }
\end{example}
\begin{example}
{\em If ${\mathcal P}$ is the class of all permutations then
  $P(x)=\sum_{n\geq 1} n!x^n$. Equation~(\ref{equation_op})
gives the ordinary generating function of the indecomposable
permutations. The numbers of these are listed as sequence A003319
in~\cite{OEIS}. 
 }
\end{example}

To refine Theorem~\ref{theorem_op} let $P_{n,k}$ denote the 
number of permutations belonging to an orbit of size $2^k$ of the
restricted Foata--Strehl group action on ${\mathcal
  P}(\{1,2,\ldots,n\})$ and let $O_{n,k}$ be the number of orbits of
size $2^k$. Clearly we have 
\begin{equation}
\label{equation_orbits}  
O_{n,k}=\frac{P_{n,k}}{2^k}. 
\end{equation}
We introduce the generating functions
$$P(x,y)=\sum_{n\geq 1, k\geq 0} P_{n,k} x^ny^k\quad\mbox{and}\quad
O(x,y)=\sum_{n\geq 1, k\geq 0} O_{n,k} x^ny^k.$$
\begin{theorem}
The generating functions $P(x,y)$ and $O(x,y)$ are given by 
\begin{align}
\label{equation_pxy}
P(x,y) & = \frac{P(x)}{1-2(y-1)P(x)} , \\
\label{equation_oxy}
O(x,y) & = \frac{P(x)}{1-(y-2)P(x)} .
\end{align}
\label{theorem_op_refined}
\end{theorem}  
\begin{proof}
We use the notation $I([\sigma])$ introduced in the  proof of
Theorem~\ref{theorem_op}. Given any $\sigma\in {\mathcal
  P}(\{1,2,\ldots,n\})$ and any element  $i_1$ of $I([\sigma])$,
the permutation $\sigma$ may be written as a concatenation of words
\begin{equation}
\sigma=\sigma_0\sigma_1\sigma_0',
\label{equation_sigma_decomposition}
\end{equation}
where $\sigma_0 i_1 \sigma_0'$ is an element of ${\mathcal
    P}(\{1,2,\ldots,i_1\})$ and $\sigma_1$ is an element of ${\mathcal
  P}(\{i_1,i_1+1,\ldots,n\})$ containing $i_1$ as the first or the last
letter. At the level of Foata--Strehl trees, $\FS(\sigma)$ may be
obtained by selecting a levelwise labeled plane $0-1-2$
tree with $i_1$ vertices and then  adding any Foata--Strehl tree with
$n-i_1$ as the right or left subtree of $i_1$.  Iterating the procedure,
for any subset $\{i_1,i_2,\ldots,i_k\}$ of $I([\sigma])$ satisfying
$i_1<i_2<\cdots<i_k$ we may decompose $\FS(\sigma)$ 
into a sequence of Foata--Strehl trees $(T_0,T_1,\ldots,T_{k})$ such that
\begin{enumerate}
\item $T_1=\FS(\sigma_0i_1\sigma_0')$ for some
  $\sigma_0i_1\sigma_0'\in {\mathcal P}(\{1,2,\ldots,i_1\})$;
\item for $j=2,3,\ldots,k-1$ the labeled tree
  $T_j=\FS(\sigma_{j-1}i_j\sigma_{j-1}')$ for some
  $\sigma_{j-1}i_j\sigma_{j-1}'\in \widehat{{\mathcal
    P}}(\{i_{j-1}+1,i_{j-1}+2,\ldots,i_j\})$;
\item $T_k=\FS(\sigma_{k-1}i_k\sigma_{k-1}')$ for some
  $\sigma_{k-1}i_k\sigma_{k-1}'\in {\mathcal
  P}(\{i_{k-1}+1,i_{k-1}+2,\ldots,i_k\})$, {\em or it may be empty};
\item for $j=2,\ldots,k$ the labeled tree $T_j$ is the right or left
  subtree of $i_j$.   
\end{enumerate}
Conversely, if $\FS(\sigma)$ is decomposed into 
into a sequence of Foata--Strehl trees $(T_0,T_1,\ldots,T_{k})$
satisfying the above criteria then $\{i_1,i_2,\ldots,i_k\}$ must be a
subset of $I([\sigma])$. Introducing the variable $y$ to mark the
selected elements of $I([\sigma])$, we obtain the formula
$$
P(x,1+y)=\sum_{n\geq 1,j\geq 0} P_{n,j} x^n(1+y)^j=
(1+P(x))\cdot \sum_{k\geq 0} (2y)^k P(x)^k=
\frac{1+P(x)}{1-2yP(x)}.
$$
Equation~(\ref{equation_pxy}) follows by replacing $y$ with $y-1$ in the
last equation. To obtain~(\ref{equation_oxy}), by
equation~(\ref{equation_orbits}), we only need to substitute $y/2$ into
$y$ in~(\ref{equation_pxy}).    
\end{proof}  
Note that Theorem~\ref{theorem_op} may be obtained by substituting
$y=1$ in~(\ref{equation_oxy}).  
\begin{example}
{\em  For Catalan--Spitzer permutations we have $P(x)=C(x)-1$ and
  \begin{eqnarray*}
   O(x,y)&=&\frac{C(x)-1}{1-(y-2)(C(x)-1)}\\
   &=& x + yx^2 + (y^2 + 1)x^3 + (y^3 + 2y + 2)x^4 + (y^4 + 3y^2 +
   4y + 6)x^5\\
   && + (y^5 + 4y^3 + 6y^2 + 13y + 18)x^6+\cdots \\ 
 \end{eqnarray*}
Substituting $y=1$ respectively $y=2$ gives $O(x)=xC(x)$, respectively
$P(x)=C(x)-1$. The substitutions $y=3,4,5,6$ are listed as sequences A001700,
A049027, A076025, A076026 in~\cite{OEIS}. The generating functions
listed for these sequences are all substitutions into $y$ in  
$\frac{1 - (y - 2)xC(x)}{1 - (y - 1)xC(x)}$ which is easily seen to be
equal to $1+O(x,y)$. 

}
\end{example}  
\begin{example}
{\em  For all permutations we have $P(x)=\sum_{n\geq 1} n! x^n$ and
  \begin{eqnarray*}
   O(x,y)&=&\frac{P(x)}{1-(y-2)P(x)}\\
   &=& x + yx^2 + (y^2 + 2)x^3 + (y^3 + 4y + 8)x^4 + (y^4 + 6y^2 +
   16y + 48)x^5\\
   && + (y^5 + 8y^3 + 24y^2 + 100y + 328)x^6+\cdots
 \end{eqnarray*}
Substituting $y=1$ respectively $y=2$ gives $O(x)$, respectively
$P(x)$. The substitution $y=3$ is listed as sequence A051296 in~\cite{OEIS}. 
The presence of this substitution is not surprising:
$O(x,3)=P(x)/(1-P(x))$ is always the generating function for the ordered
collections of permutations in the class ${\mathcal P}$. }
\end{example}

\section{Concluding remarks}

Are there any other results on the distribution of the quantity $\alpha$ for more
for general paths from the origin $(0,0)$ to $(s,0)$?
For instance, what can be said if we have one type of up step,
but two types of down steps?

Are there other Fuss--Catalan structures that belongs
to a larger set of structures of cardinality
$\binom{kn+1}{n}$ with a uniformly distributed statistic
on the set $\{0,1, \ldots, kn\}$
and the Fuss--Catalan structure is the fiber of one particular
value of this statistic?

\section*{Acknowledgments}

This work was partially supported by grants from the
Simons Foundation
(\#429370 to Richard~Ehrenborg,
\#245153 and \#514648 to G\'abor~Hetyei,
\#422467 to Margaret~Readdy). Margaret Readdy was also supported by 
NSF grant DMS-2247382.

\newcommand{\journal}[6]{{\sc #1,} #2, {\it #3} {\bf #4} (#5), #6.}
\newcommand{\book}[4]{{\sc #1,} ``#2,'' #3, #4.}
\newcommand{\bookf}[5]{{\sc #1,} ``#2,'' #3, #4, #5.}
\newcommand{\arxiv}[3]{{\sc #1,} #2, {\tt #3}.}
\newcommand{\preprint}[3]{{\sc #1,} #2, preprint {(#3)}.}
\newcommand{\preparation}[2]{{\sc #1,} #2, in preparation.}
\newcommand{\toappear}[3]{{\sc #1,} #2, to appear in {\it #3}.}

\end{document}